\documentclass[11pt]{article}
\usepackage{}
\usepackage{mathrsfs}
\usepackage{amsfonts}
\usepackage{amsfonts,amssymb,mathrsfs}
\usepackage{amsmath,amscd}
\usepackage[dvips]{graphicx}
\usepackage[all]{xy}
\usepackage{graphicx}
\usepackage{fancyhdr}
\usepackage{mathptmx,pslatex}
\usepackage{amsthm}
\usepackage{rotating}
\usepackage{mathdots}

\usepackage{xcolor}
\usepackage{listings}
\usepackage{tikz}
\begin{document}

\sloppy
\newtheorem{Def}{Definition}[section]
\newtheorem{Bsp}{Example}[section]
\newtheorem{Prop}[Def]{Proposition}
\newtheorem{Theo}[Def]{Theorem}
\newtheorem{Lem}[Def]{Lemma}
\newtheorem{Koro}[Def]{Corollary}
\theoremstyle{definition}
\newtheorem{Rem}[Def]{Remark}

\newcommand{\add}{{\rm add}}
\newcommand{\gd}{{\rm gl.dim }}
\newcommand{\dm}{{\rm dom.dim }}
\newcommand{\E}{{\rm E}}
\newcommand{\Mor}{{\rm Morph}}
\newcommand{\End}{{\rm End}}
\newcommand{\ind}{{\rm ind}}
\newcommand{\rsd}{{\rm res.dim}}
\newcommand{\rd} {{\rm rep.dim}}
\newcommand{\ol}{\overline}
\newcommand{\rad}{{\rm rad}}
\newcommand{\soc}{{\rm soc}}
\renewcommand{\top}{{\rm top}}
\newcommand{\pd}{{\rm proj.dim}}
\newcommand{\id}{{\rm inj.dim}}
\newcommand{\Fac}{{\rm Fac}}
\newcommand{\fd} {{\rm fin.dim }}
\newcommand{\DTr}{{\rm DTr}}
\newcommand{\cpx}[1]{#1^{\bullet}}
\newcommand{\D}[1]{{\mathscr D}(#1)}
\newcommand{\Dz}[1]{{\mathscr D}^+(#1)}
\newcommand{\Df}[1]{{\mathscr D}^-(#1)}
\newcommand{\Db}[1]{{\mathscr D}^b(#1)}
\newcommand{\C}[1]{{\mathscr C}(#1)}
\newcommand{\Cz}[1]{{\mathscr C}^+(#1)}
\newcommand{\Cf}[1]{{\mathscr C}^-(#1)}
\newcommand{\Cb}[1]{{\mathscr C}^b(#1)}
\newcommand{\K}[1]{{\mathscr K}(#1)}
\newcommand{\Kz}[1]{{\mathscr K}^+(#1)}
\newcommand{\Kf}[1]{{\mathscr  K}^-(#1)}
\newcommand{\Kb}[1]{{\mathscr K}^b(#1)}
\newcommand{\modcat}{\ensuremath{\mbox{{\rm -mod}}}}
\newcommand{\Modcat}{\ensuremath{\mbox{{\rm -Mod}}}}
\newcommand{\stmodcat}[1]{#1\mbox{{\rm -{\underline{mod}}}}}
\newcommand{\pmodcat}[1]{#1\mbox{{\rm -proj}}}
\newcommand{\imodcat}[1]{#1\mbox{{\rm -inj}}}
\newcommand{\opp}{^{\rm op}}
\newcommand{\otimesL}{\otimes^{\rm\bf L}}
\newcommand{\rHom}{{\rm\bf R}{\rm Hom}}
\newcommand{\projdim}{\pd}
\newcommand{\Hom}{{\rm Hom}}
\newcommand{\Coker}{{\rm coker}\,\,}
\newcommand{ \Ker  }{{\rm Ker}\,\,}
\newcommand{ \Img  }{{\rm Im}\,\,}
\newcommand{\Ext}{{\rm Ext}}
\newcommand{\StHom}{{\rm \underline{Hom} \, }}

\newcommand{\gm}{{\rm _{\Gamma_M}}}
\newcommand{\gmr}{{\rm _{\Gamma_M^R}}}

\def\vez{\varepsilon}\def\bz{\bigoplus}  \def\sz {\oplus}
\def\epa{\xrightarrow} \def\inja{\hookrightarrow}

\newcommand{\lra}{\longrightarrow}
\newcommand{\lraf}[1]{\stackrel{#1}{\lra}}
\newcommand{\ra}{\rightarrow}
\newcommand{\dk}{{\rm dim_{_{k}}}}

\newcommand{\rank}{{\rm rank}}

\allowdisplaybreaks[4]

{\Large \bf
\begin{center}
The Smith normal form of the $Q$-walk matrix of the Dynkin graph $A_n$
\end{center}}

\medskip

\centerline{{\bf Yaning Jia, Shengyong Pan$^*$} }
\begin{center} School of Mathematics and Statistics, Beijing Jiaotong University, \\
100044 Beijing, People's Republic of  China \\ E-mail:  22121587@bjtu.edu.cn, shypan@bjtu.edu.cn \\
\end{center}
\bigskip

\renewcommand{\thefootnote}{\alph{footnote}}
\setcounter{footnote}{-1} \footnote{ $^*$ Corresponding author.
Email: shypan@bjtu.edu.cn; Tel.: 0086 10
51685464.}
\renewcommand{\thefootnote}{\alph{footnote}}
\setcounter{footnote}{-1} 
\renewcommand{\thefootnote}{\alph{footnote}}
\setcounter{footnote}{-1} 

\date{}

\begin{abstract}
In this paper, we give an explicit formula for the rank of the $Q$-walk matrix of the Dynkin graph $A_n$. Moreover, we prove that its Smith normal form is
$$
\mathrm{diag}\left( \underset{r=\lceil \frac{n}{2} \rceil}{\underbrace{1,2,2,...,2}},0,...,0 \right),
$$
where $r$ is the rank of the $Q$-walk matrix $W_Q\left( A_n \right) $ of the Dynkin graph $A_n$.
\end{abstract}

\tableofcontents

\section{Introduction}

Let $M$ be an $n\times n$ integral matrix of order $n$. For $1\leqslant i\leqslant n$, the $i$-th determinant divisor of $M$ , denoted by  $D(i)$, is the greatest common divisor of all minors of order $i$ of $M$. Suppose that $\rank (M)=r$, then the determinant divisors satisfy $D(i-1) |D(i)$ for $i=1,...,r$, where $D(0):=1$. The invariant factors of $M$ are defined by
$$
d_1=\frac{D\left( 1 \right)}{D\left( 0 \right)},d_2=\frac{D\left( 2 \right)}{D\left( 1 \right)},...,d_r=\frac{D\left( r \right)}{D\left( r-1 \right)}.
$$
It is well known that each $d_i$ divides $d_{i+1}$ for $i=1,...,r-1$. Then the \emph{Smith normal form} of $M$ is the diagnoal matrix
$$
\mathrm{diag}\left( d_1,...,d_r,0,...,0 \right).
$$
The \emph{walk matrix} $W(M)$ of $M$ is defined by
$$
W\left( M \right) =\left[ e_n\,,Me_n\,,\cdots ,\,M^{n-1}e_n \right],
$$
where $e_n$ is the all-ones vector of demension $n$.

Let $G$ be a simple graph with vertex set $\left\{ 1,...,n \right\}$ and let $A$ be the adjacency matrix of $G$. The \emph{adjacency walk matrix} of $G$ is
$$
W_A=W_A\left( G \right) :=\left[ e_n\,, Ae_n\,, \cdots \,, A^{n-1}e_n \right],
$$
where $e_n$ is the all-ones vector of demension $n$. It is also called the walk matrix of $G$, denoted by $W_A(G)$. Note that the $(i,j)$-entry of the walk matrix $W_A(G)$ counts the number of walks in $G$ of length $j-1$ from the vertex $i$.
In \cite{ref13}, Wang, Wang and Guo gave a formula for the rank of the walk matrix of Dynkin graph $D_n$ for $n\geq 4$. Moreover, they provided the Smith normal form of the walk matrix of the Dynkin graph $D_n$ for the case $n\not\equiv 0 (\text{mod}\;4)$. Later in \cite{ref14}, they gave the Smith normal form of the walk matrix of the Dynkin graph $D_n$ for the case $n\equiv 0 (\text{mod}\;4)$. In \cite{ref3}, S. Moon and S. Park have provided formulas for the rank and the Smith normal form of the walk matrix of the extended Dynkin graph $\widetilde{D}_n$. Recently, the Smith normal form of the walk matrix of the Dynkin graph $A_n$ was given in \cite{ref2}. Dynkin graphs and extended Dynkin graphs are widely used in the study of simple Lie algebras in Lie theory \cite{ref4,ref5}, the classifications of 
representation-finite hereditary algebras in  representation theory \cite{ref6,ref7} and spectral theory \cite{ref8,ref9}.

Let $G$ be a simple graph with vertex set $\left\{ 1,...,n \right\}$.
The \emph{signless Laplacian matrix} of $G$, denoted by $Q(G)$, is $A(G)+D(G)$, where $A(G)$ is the the adjacency matrix of $G$ and $D(G)$ is the the degree diagonal matrix of $G$ (see \cite{ref1}). The \emph{signless Laplacian matrix} (or \emph{$Q$-walk matrix}) of $G$, denoted by $W_Q\left( G \right) $, is naturally defined by
$$
W_Q(G) =\left[ e_n\,,Q(G)e_n\,,\cdots \,,Q(G)^{n-1}e_n \right].
$$
In \cite{ref1}, Yan, Yao and Wang studied the determinant of the $Q$-walk matrix of rooted product with a path. In \cite{QYJ}, Qiu, Yuan and Ji, showed that at most  $r=\lfloor \frac{n}{2} \rfloor $ invariant factors of $Q$-walk matrix of a graph with $n$ vertices are congruent to $4$ modulo $8$, moreover, 
if $2^{-\lceil \frac{5n-4r-22}{2} \rceil}\det W_Q$ is odd and square-free, then the Smith normal form of $W_Q$ can be determined uniquely by the triple $(n,r,\det W_Q)$.
Therefore, it is interesting to calculate 
the Smith normal form of the $Q$-walk matrix of Dynkin graphs.

In this paper, let $n$ be a fixed positive integer, we consider the Dynkin graph $A_n$ with vertex set $\left\{ 1,...,n \right\}$
\begin{center}
\begin{tikzpicture}
	\draw[line width=1pt](0,1)--(3,1);
	\draw[line width=1pt](3,1)--(5,1);
	\draw[dotted,line width=1pt](5,1)--(7,1);
	\draw[line width=1pt](7,1)--(9,1)--(12,1);
	\fill(0,1) circle(.05) node[below] {$1$};
	\fill(3,1) circle(.05)  node[below] {$2$};
	\filldraw(9,1) circle(.05) node[below] {$n-1$};
	\filldraw(12,1) circle(.05) node[below] {$n$};
\end{tikzpicture}
\end{center}
and give the Smith normal form of its $Q$-walk matrix.

\section{Preliminaries\label{pre}}

Unless otherwise specified, we set $r=\lceil \frac{n}{2} \rceil $. We consider the submatrix $\overline{W_Q\left( A_n \right) }$ of the $Q$-walk matrix $
W_Q\left( A_n \right) $ obtained by deleting\\
(1) the $(r+1)$-th to the $n$-th rows of $W_Q\left( A_n \right) $, and\\
(2) the $(r+1)$-th to the $n$-th columns of $W_Q\left( A_n \right) $.

In the following, we give two examples of $W_Q\left( A_n \right) $ and $\overline{W_Q\left( A_n \right) }$ for $n=3,10$.

\begin{Bsp}
The $Q$-walk matrix $W_Q(A_3) $ is
 $$
\left[ \begin{matrix}
 	1&		2&		6\\
 	1&		4&		12\\
 	1&		2&		6\\
 \end{matrix} \right]
 $$
and
 $$
 \overline{W_Q\left( A_3 \right) }=\left[ \begin{matrix}
 	1&		2\\
 	1&		4\\
 \end{matrix} \right] .
 $$
 The $Q$-walk matrix $W_Q\left( A_{10} \right) $ is
 $$
 	\left[ \begin{matrix}
 	1&		2&		6&     20&    70&     252&     924& 3432&   12870&  48620\\
 	1&		4&		14&    50&    182&    672&     2508& 9438&   35750&  136134\\
 	1&		4&		16&    62&    238&    912&     3498& 13442&  51764&  199746\\
 	1&      4&      16&    64&     254&     1002& 3938&   15442&   60468& 236568\\
 	1&     4&  16& 64& 256& 1002& 4068& 16142& 63868& 252072\\
 	1&     4& 16& 64& 256& 1002& 4068& 16142& 63868& 252072\\
 	1&     4& 16& 64& 254& 1002& 3938& 15442& 60468& 236568\\
 	1&		4&		16& 62& 238& 912& 3498& 13442& 51764& 199746\\
 	1&		4&		14& 50& 182& 672& 2508& 9438& 35750& 136134\\
 	1&		2&		6&    20&  70&  252&  924& 3432& 12870& 48620\\
 \end{matrix} \right]
 $$
 \\and
 $$
 \overline{W_Q\left( A_{10} \right) }=\left[ \begin{matrix}
 	1&		2&		6&     20&    70\\
 	1&		4&		14&    50&    182\\
 	1&		4&		16&    62&    238\\
 	1&      4&      16&    64&     254\\
 	1&      4&       16&   64&    256\\
 \end{matrix} \right] .
 $$
\end{Bsp}
We consider two partitions for the vertex set of the Dynkin graph $A_n$. The following partitions with r cells
$$
\varPi _1=\left\{ \left\{ 1,n \right\} ,\left\{ 2,n-1 \right\} ,...,\left\{ \frac{n}{2},\frac{n}{2}+1 \right\} \right\}
$$
\\and
$$
\varPi _2=\left\{ \left\{ 1,n \right\} ,\left\{ 2,n-1 \right\} ,...,\left\{ \frac{n-1}{2},\frac{n+3}{2} \right\} ,\left\{ \frac{n+1}{2} \right\} \right\}
$$
are equitable for $n$ even and odd, respectively(see \cite{ref2}). Note that for any graph $G$, the rank of the $Q$-walk matrix $
W_Q\left(G \right) $ is less than or equal to the number of cells in any equitable partitions of the vertex set of $G$(see \cite{ref10,ref11}). Thus we have rank $W_Q\left( A_n \right) \leqslant r$.

For any positive integer $m$ and any square matrix $M$, we set 
$M\oplus O_m=\left[ \begin{matrix}
	M&		\\
	&		O_m\\
\end{matrix} \right]
$, where $O_m$ is the square zero matrix of order $m$. In the following lemma, we show that $W_Q\left(A_n \right) $ and $\overline{W_Q\left( A_n \right) }\oplus O_{n-r}$ have the same Smith normal form.
\begin{Lem}\label{2.1}
The matrices $W_Q\left(A_n \right) $ and $\overline{W_Q\left( A_n \right) }\oplus O_{n-r}$ have the same Smith normal form, where $r=\lceil \frac{n}{2} \rceil
$.
\end{Lem}

\begin{proof}
    
Let rank $W_Q\left( A_n \right) =t$. Then  $t\leqslant r$. Since the first $t$ columns of $W_Q\left( A_n \right)$ are linearly independent (see \cite{ref11,ref12}), they form a basis for the column space of $W_Q\left( A_n \right)$. Thus the $i$-th column can be written as a linear combination of
$$
e_n,Q(G)e_n,...,Q(G)^{t-1}e_n
$$
for $r+1\leqslant i\leqslant n$. Note that the $i$-th row and the $(n+1-i)$-th row are identical for $i=1,...,r$. Hence we obtain $\overline{W_Q\left( A_n \right) }\oplus O_{n-r}$ by using the elementary row and column operations on $W_Q\left(A_n \right) $. Therefore $W_Q\left(A_n \right) $ and $\overline{W_Q\left( A_n \right) }\oplus O_{n-r}$ have the same Smith normal form.
\end{proof} 

Let
\begin{center}
$C_1=\left[ \begin{matrix}
		1 &		\\
		&\ddots	\\
		& & 1\\
		& & & 1\\
		& & & 1\\
		& & 1\\
		& \iddots & & \\
		1 &  \\
\end{matrix} \right]_{n\times r}
$ and   
$C_2=\left[ \begin{matrix}
		1 &		\\
		&\ddots	\\
		& & 1\\
		& & & 1\\
		& & 1\\
		& \iddots & & \\
		1 &  \\
	\end{matrix} \right]_{n\times r}
	$
\end{center}
be the characteristic matrices of the partitions $
\varPi _1$ and $\varPi _2$, respectively. The divisor matrices of $\varPi _1$ and $\varPi _2$ are
\begin{center}
$B_1=\left[ \begin{matrix}
	0& 1&		\\
	1& 0& 1 \\
	&\ddots& \ddots& \ddots&	\\
	& & 1& 0& 1\\
	& &  & 1& 1\\
\end{matrix} \right]_{r\times r}
$ and $B_2=\left[ \begin{matrix}
	0& 1&		\\
	1& 0& 1&  & \\
	&\ddots& \ddots& \ddots&  &	\\
	& & 1& 0& 1\\
	& & & 2& 0\\
\end{matrix} \right]_{r\times r}
$,
\end{center}
respectively(see \cite{ref2}).

By calculation and observation, it is easy to see that the degree matrix  $D\left( A_n \right) $ of the Dynkin graph $A_n$ has the following form
$$D\left( A_n \right)=\left[ \begin{matrix}
	1 &		\\
	& 2&  &	\\
	& & 2& & \\
	& & & \ddots& & \\
	& & & & 2& &\\
    &  & & & & 1\\
\end{matrix} \right]_{n\times n}.
$$
Let $\overline{D\left( A_n \right) }$ be the $r\times r$ submatrix obtained from the $D\left( A_n \right) $ by deleting\\
$(1)$ the $(r+1)$-th to the $n$-th rows of $D\left( A_n \right) $, and
\\
$(2)$ the $(r+1)$-th to the $n$-th columns of $D\left( A_n \right) $.

In the following lemma, we give the relation between $\overline{W_Q(A_n)}$ and $W(B_i+\overline{D(A_n)})$ for $i=1,2$.

\begin{Lem}\label{2.2}
$(a)$ If $n$ is even, then $
	\overline{W_Q\left( A_n \right) }=W\left( B_1+\overline{D\left( A_n \right) } \right)
	$.
    
$(b)$ If $n$ is odd, then $
	\overline{W_Q\left( A_n \right) }=W\left( B_2+\overline{D\left( A_n \right) } \right)
	$.
 \end{Lem}	
 \begin{proof}
  We only prove $(a)$ and the proof of $(b)$
is similar. Since $D\left( A_n \right) C_1=C_1\overline{D\left( A_n \right) }$ and $AC_1=C_1B_1$, we have $(A+D(A_n)) C_1=C_1( B_1+\overline{D( A_n)}) $. Then we obtain $\left( A+D\left( A_n \right) \right) ^kC_1=C_1\left( B_1+\overline{D\left( A_n \right) } \right) ^k$ for all $k\geqslant 0$. Since $C_1e_r=e_n$, we have $
( A+D(A_n))^ke_n=\left( A+D\left( A_n \right) \right) ^kC_1e_r=C_1\left( B_1+\overline{D\left( A_n \right) } \right) ^ke_r$ for all $k\geqslant 0$. It follows that
$$
\left[ e_n,( A+D( A_n)e_n, 	\cdots,( A+D(A_n))^{r-1}e_n\right] =
C_1\left[e_r,( B_1+\overline{D( A_n)})e_r,\cdots,( B_1+\overline{D( A_n)}) ^{r-1}e_r\right].
$$
If we delete the $(r+1)$-th to the $n$-th rows of $
\left[ e_n, ( A+D( A_n))e_n,\cdots,	(A+D( A_n))^{r-1}e_n \right] $ and $C_1$, we get $
\overline{W( A+D( A_n))}$ and the identity matrix, respectively. So we have $
\overline{W(A+D(A_n)) }=W\left( B_1+\overline{D\left( A_n \right) } \right)
$. Hence, $\overline{W_Q\left( A_n \right) }=W\left( B_1+\overline{D\left( A_n \right) } \right) $.
 \end{proof}

The following lemma provides a method for computing the determinant and the rank of a walk matrix.

\begin{Lem}\cite{ref13}\label{2.3}
Let M be a real $n\times n$ matrix which is diagonalizable over the real field $\mathbb{R}$. Let $\xi _1,\xi _2,...,\xi _n$ be $n$ linearly independent eigenvectors of $
M^T$ corresponding to eigenvalues $\lambda _1,\lambda _2,...,\lambda _n$, respectively. Then we have
$$
\det W(M) =\frac{\prod_{1\leqslant k<j\leqslant n}\left( \lambda _j-\lambda _k \right) \prod_{j=1}^n{e_{n}^{T}\xi _j}}{\det \left[ \begin{matrix}
		\xi _1&		\xi _2&		\cdots&		\xi _n\\
	\end{matrix} \right]}.
$$
\\Moreover, if $\lambda _1,\lambda _2,...,\lambda _n$ are pairwise different, then
$$
\rank \;W(M) =\left| \left\{ j|e_{n}^{T}\xi _j\ne 0 \,\text{and}\,\,j=1,...,n \right\} \right|.
$$
\end{Lem}	

\section{Main results\label{main}}
We will give the main results of this paper in this section.

\subsection{Smith normal form of $W_Q\left( A_n \right) $ for $n$ is even}

In the following proposition, we find the eigenvalues and corresponding eigenvectors of $
\left( B_1+\overline{D\left( A_n \right) } \right) ^T
$.

\begin{Prop}
	Let $\lambda _k=2-2\cos \alpha _k$ and let
	$$
	v_k=\left[ \begin{array}{c}
		\left( -1 \right) ^{r-1}\left( 1+\sum_{i=1}^{r-1}{2\cos i\alpha _k} \right)\\
		\left( -1 \right) ^{r-2}\left( 1+\sum_{i=1}^{r-2}{2\cos i\alpha _k} \right)\\
		\vdots\\
		-1-2\cos \alpha _k\\
		1\\
	\end{array} \right] ,
	$$
where $\alpha _k=\frac{2k-1}{2r}\pi $  for $k=1,...,r$. Then $v_k$ is an eigenvector of $
\left( B_1+\overline{D\left( A_n \right) } \right) ^T
$ corresponding to the eigenvalue $\lambda _k$ for $k=1,...,r$.
\end{Prop}
\begin{proof}
It is enough to show that $
\left( B_1+\overline{D\left( A_n \right) } \right) ^Tv_k=\lambda _kv_k
$ for $k=1,...,r$. First, we show that the first entries of $\left( B_1+\overline{D\left( A_n \right) } \right) ^Tv_k$ and $\lambda _kv_k$ are the same, that is,
$$
\left( -1 \right) ^{r-1}\left( 1+\sum_{i=1}^{r-1}{2\cos i\alpha _k} \right) +\left( -1 \right) ^{r-2}\left( 1+\sum_{i=1}^{r-2}{2\cos i\alpha _k} \right) =\left( -1 \right) ^{r-1}\left( 2-2\cos \alpha _k \right) \left( 1+\sum_{i=1}^{r-1}{2\cos i\alpha _k} \right).
$$
The first entry of $\left( B_1+\overline{D\left( A_n \right) } \right) ^Tv_k$ is
$$\left( -1 \right) ^{r-1}\left( 1+\sum_{i=1}^{r-1}{2\cos i\alpha _k} \right) +\left( -1 \right) ^{r-2}\left( 1+\sum_{i=1}^{r-2}{2\cos i\alpha _k} \right)$$
\begin{equation*}
	\begin{split}
		&=\left( -1 \right) ^{r-1}\left( 1+\sum_{i=1}^{r-1}{2\cos i\alpha _k} \right) +\left( -1 \right) ^{r-1}\left( -1-\sum_{i=1}^{r-2}{2\cos i\alpha _k} \right)\\
		&=\left( -1 \right) ^{r-1}2\cos \left( r-1 \right) \alpha _k .
	\end{split}
\end{equation*}
The first entry of $\lambda _kv_k$ is
\begin{equation*}
	\begin{split}
		&\left( -1 \right) ^{r-1}\left( 2-2\cos \alpha _k \right) \left( 1+\sum_{i=1}^{r-1}{2\cos i\alpha _k} \right)\\
		&=-2\left( -1 \right) ^{r-2}\left( 1-\cos \alpha _k \right) \left( 1+\sum_{i=1}^{r-1}{2\cos i\alpha _k} \right) \\
		&=-2\left( -1 \right) ^{r-2}\left[ \left( 1+\sum_{i=1}^{r-1}{2\cos i\alpha _k} \right) -\left( \cos \alpha _k+\sum_{i=1}^{r-1}{2\cos \alpha _k\cos i\alpha _k} \right) \right] .
	\end{split}
\end{equation*}
\\By applying the formula $
2\cos \alpha _k\cos i\alpha _k=\cos \left( i+1 \right) \alpha _k+\cos \left( i-1 \right) \alpha _k
$ for $i=1,...,r-1$, we obtain
\begin{equation*}
	\begin{split}
		& \cos \alpha _k+\sum_{i=1}^{r-1}{2\cos \alpha _k\cos i\alpha _k}\\
		&=\cos \alpha _k+\sum_{i=1}^{r-1}{\left( \cos \left( i+1 \right) \alpha _k+\cos \left( i-1 \right) \alpha _k \right)}\\
		&=1+\sum_{i=1}^r{\cos i\alpha _k}+\sum_{i=1}^{r-2}{\cos i\alpha _k} .
	\end{split}
\end{equation*}
It follows that 
\begin{equation*}
	\begin{split}
		&\left( -1 \right) ^{r-1}\left( 2-2\cos \alpha _k \right) \left( 1+\sum_{i=1}^{r-1}{2\cos i\alpha _k} \right)\\
		&=-2\left( -1 \right) ^{r-2}\left[ \left( 1+\sum_{i=1}^{r-1}{2\cos i\alpha _k} \right) -\left( 1+\sum_{i=1}^r{\cos i\alpha _k}+\sum_{i=1}^{r-2}{\cos i\alpha _k} \right) \right]\\
		&=-2\left( -1 \right) ^{r-2}\left( \cos \left( r-1 \right) \alpha _k-\cos r\alpha _k \right) .
\end{split}
\end{equation*}
\\Since $\cos r\alpha _k=\cos \frac{2k-1}{2}\pi =0$, it follows that
$$-2\left( -1 \right) ^{r-2}\cos \left( r-1 \right) \alpha _k=\left( -1 \right) ^{r-1}2\cos \left( r-1 \right) \alpha _k .$$
\\Hence the first entries of $\left( B_1+\overline{D\left( A_n \right) } \right) ^Tv_k$ and $\lambda _kv_k$ are the same.

Next, we prove that the $j$-th entries of $\left( B_1+\overline{D\left( A_n \right) } \right) ^Tv_k$ and $\lambda _kv_k$ are equal for $j=2,...,r-1$.
\\The $j$-th entry of $\left( B_1+\overline{D\left( A_n \right) } \right) ^Tv_k$  is
\begin{equation*}
	\begin{split}
	&\left( -1 \right) ^{r+1-j}\left( 1+\sum_{i=1}^{r+1-j}{2\cos i\alpha _k} \right) +2\left( -1 \right) ^{r-j}\left( 1+\sum_{i=1}^{r-j}{2\cos i\alpha _k} \right) +\left( -1 \right) ^{r-1-j}\left( 1+\sum_{i=1}^{r-1-j}{2\cos i\alpha _k} \right)\\
	&=\left( -1 \right) ^{r+1-j}\left( 1+\sum_{i=1}^{r+1-j}{2\cos i\alpha _k} \right) -2\left( -1 \right) ^{r+1-j}\left( 1+\sum_{i=1}^{r-j}{2\cos i\alpha _k} \right) +\left( -1 \right) ^{r+1-j}\left( 1+\sum_{i=1}^{r-1-j}{2\cos i\alpha _k} \right) \\
	&=\left( -1 \right) ^{r+1-j}\left( \sum_{i=1}^{r+1-j}{2\cos i\alpha _k}+\sum_{i=1}^{r-1-j}{2\cos i\alpha _k}-2\sum_{i=1}^{r-j}{2\cos i\alpha _k} \right) \\
	&=\left( -1 \right) ^{r+1-j}\left( 2\cos \left( r+1-j \right) \alpha _k-2\cos \left( r-j \right) \alpha _k \right) .	
	\end{split}
\end{equation*}
\\The $j$-th entry of $\lambda _kv_k$ is
\begin{equation*}
	\begin{split}
		&\left( -1 \right) ^{r-j}\left( 2-2\cos \alpha _k \right) \left( 1+\sum_{i=1}^{r-j}{2\cos i\alpha _k} \right) \\
		&=-2\left( -1 \right) ^{r+1-j}\left( 1-\cos \alpha _k \right) \left( 1+\sum_{i=1}^{r-j}{2\cos i\alpha _k} \right) \\
		&=-2\left( -1 \right) ^{r+1-j}\left[ \left( 1+\sum_{i=1}^{r-j}{2\cos i\alpha _k} \right) -\left( \cos \alpha _k+\sum_{i=1}^{r-j}{2\cos \alpha _k\cos i\alpha _k} \right) \right].
	\end{split}
\end{equation*}
\\Similarly, by applying the formula $
2\cos \alpha _k\cos i\alpha _k=\cos \left( i+1 \right) \alpha _k+\cos \left( i-1 \right) \alpha _k
$ for $i=1,...,r-j$, we obtain
\begin{equation*}
	\begin{split}
		&\left( -1 \right) ^{r-j}\left( 2-2\cos \alpha _k \right) \left( 1+\sum_{i=1}^{r-j}{2\cos i\alpha _k} \right) \\
		&=-2\left( -1 \right) ^{r+1-j}\left[ \left( 1+\sum_{i=1}^{r-j}{2\cos i\alpha _k} \right) -\left( 1+\sum_{i=1}^{r-1-j}{\cos i\alpha _k}+\sum_{i=1}^{r+1-j}{\cos i\alpha _k} \right) \right] \\
		&=-2\left( -1 \right) ^{r+1-j}\left( \cos \left( r-j \right) \alpha _k-\cos \left( r+1-j \right) \alpha _k \right) \\
		&=\left( -1 \right) ^{r+1-j}\left( 2\cos \left( r+1-j \right) \alpha _k-2\cos \left( r-j \right) \alpha _k \right) .
	\end{split}
\end{equation*}
\\Thus the $j$-th entries of $\left( B_1+\overline{D\left( A_n \right) } \right) ^Tv_k$ and $\lambda _kv_k$ are the same for $j=2,...,r-1$.
\\Since
$$
1\times \left( -1-2\cos \alpha _k \right) +3\times 1=2-2\cos \alpha _k=\left( 2-2\cos \alpha _k \right) \times 1,
$$
the last entries of $\left( B_1+\overline{D\left( A_n \right) } \right) ^Tv_k$ and $\lambda _kv_k$ are also equal. Hence, $
\left( B_1+\overline{D\left( A_n \right) } \right) ^Tv_k=\lambda _kv_k
$ for $k=1,...,r$.
\end{proof}

\par In order to find the determinant of $W\left( B_1+\overline{D\left( A_n \right) } \right) $, we first compute the product $\prod_{k=1}^r{e_{r}^{T}v_k}$.

\begin{Lem}\label{3.2}
Let $e_r$ be the all-ones vector of dimension $r$, and let
	$$
	v_k=\left[ \begin{array}{c}
		\left( -1 \right) ^{r-1}\left( 1+\sum_{i=1}^{r-1}{2\cos i\alpha _k} \right)\\
		\left( -1 \right) ^{r-2}\left( 1+\sum_{i=1}^{r-2}{2\cos i\alpha _k} \right)\\
		\vdots\\
		-1-2\cos \alpha _k\\
		1\\
	\end{array} \right] ,
	$$
where $\alpha _k=\frac{2k-1}{2r}\pi $  for $k=1,...,r$. Then
	$$
	\prod_{k=1}^r{e_{r}^{T}v_k}=\left( -1 \right) ^{\lfloor \scriptstyle{\frac{r}{2}} \rfloor}2^{r-1}.
	$$
Furthermore, $\rank \;W\left( B_1+\overline{D\left( A_n \right) } \right) =r$.
\end{Lem}

\begin{proof}
Firstly, we show that
$$
\prod_{k=1}^r{e_{r}^{T}v_k}=\prod_{k=1}^r{\frac{\sin r\alpha _k}{\sin \alpha _k}}.
$$
If $r$ is odd, then
$$
e_{r}^{T}v_k=1+2\cos 2\alpha _k+2\cos 4\alpha _k+\cdots +2\cos \left( r-3 \right) \alpha _k+2\cos \left( r-1 \right) \alpha _k.
$$
Indeed, by $
2\cos \alpha \sin \beta =\sin \left( \alpha +\beta \right) -\sin \left( \alpha -\beta \right)
$, we obtain that
\begin{equation*}
	\begin{split}
&\sin \alpha _ke_{r}^{T}v_k=\sin \alpha _k\left( 1+2\cos 2\alpha _k+2\cos 4\alpha _k+\cdots +2\cos \left( r-3 \right) \alpha _k+2\cos \left( r-1 \right) \alpha _k \right) \\
&=\sin \alpha _k+\left( \sin 3\alpha _k-\sin \alpha _k \right) +\cdots +\left( \sin r\alpha _k-\sin \left( r-2 \right) \alpha _k \right) \\
&=\sin r\alpha _k.
    \end{split}
\end{equation*}
Note that $\alpha _k=\frac{2k-1}{2r}\pi \in \left( 0,\pi \right) $ for $k=1,...,r$, so $\sin \alpha _k\ne 0$. Hence
$$
\prod_{k=1}^r{e_{r}^{T}v_k}=\prod_{k=1}^r{\frac{\sin r\alpha _k}{\sin \alpha _k}}.
$$
If $r$ is even, then
$$
e_{r}^{T}v_k=-2\cos \alpha _k-2\cos 3\alpha _k-\cdots -2\cos \left( r-3 \right) \alpha _k-2\cos \left( r-1 \right) \alpha _k.
$$
Similarly,
$$\sin \alpha _ke_{r}^{T}v_k=\sin \alpha _k\left(-2\cos \alpha _k-2\cos 3\alpha _k-\cdots -2\cos \left( r-3 \right) \alpha _k-2\cos \left( r-1 \right) \alpha _k \right)=-\sin r\alpha _k.$$
In this case, we have
$$
\prod_{k=1}^r{e_{r}^{T}v_k}=\prod_{k=1}^r{\frac{\left( -1 \right) ^r\sin r\alpha _k}{\sin \alpha _k}=}\prod_{k=1}^r{\frac{\sin r\alpha _k}{\sin \alpha _k}}.
$$
So the above equation holds.

Since $\sin r\alpha _k=\sin \frac{2k-1}{2}\pi =-\cos k\pi $, we have
$$
\prod_{k=1}^r{e_{r}^{T}v_k}=\prod_{k=1}^r{\frac{\sin r\alpha _k}{\sin \alpha _k}}=\prod_{k=1}^r{\frac{\left( -\cos k\pi \right)}{\sin \alpha _k}}=\frac{\prod_{k=1}^r{\left( -\cos k\pi \right)}}{\prod_{k=1}^r{\sin \alpha _k}}.
$$
Note that
\begin{center}
	$
	\prod_{k=1}^r{\left( -\cos k\pi \right)}=\left( -1 \right) ^{\lfloor \frac{r}{2} \rfloor}
	$ and $
	\prod_{k=1}^r{\sin \alpha _k}=\left( \frac{1}{2} \right) ^{r-1}.
	$
\end{center}
The specific calculations are given below. We first calculate $\prod_{k=1}^r{\left( -\cos k\pi \right)}$.
Note that
$$\prod_{k=1}^r{\left( -\cos k\pi \right)}=\left( -1 \right) ^r\prod_{k=1}^r{\cos k\pi}.$$
If $r$ is even, then
 $$\prod_{k=1}^r{\left( -\cos k\pi \right)}=\left( -1 \right) ^r\cdot \left( -1 \right) ^{\frac{r}{2}}=\left( -1 \right) ^{\frac{r}{2}}=\left( -1 \right) ^{\lfloor \frac{r}{2} \rfloor}.$$
If $r$ is odd, then
$$\prod_{k=1}^r{\left( -\cos k\pi \right)}=\left( -1 \right) ^r\cdot \left( -1 \right) ^{\lfloor \frac{r}{2} \rfloor +1}=\left( -1 \right) \cdot \left( -1 \right) ^{\lfloor \frac{r}{2} \rfloor}\cdot \left( -1 \right) =\left( -1 \right) ^{\lfloor \frac{r}{2} \rfloor}.$$
So we have
$$
\prod_{k=1}^r{\left( -\cos k\pi \right)}=\left( -1 \right) ^{\lfloor \frac{r}{2} \rfloor}.
$$
Now we calculate  $\prod_{k=1}^r{\sin \alpha _k}$.
First, we show that
$$
\sin \left( n\theta \right) =2^{n-1}\prod_{k=0}^{n-1}{\sin \left( \theta +\frac{k\pi}{n} \right)}.
$$
Note that $$\sin \left( n\theta \right) =\frac{e^{in\theta}-e^{-in\theta}}{2i}
=\frac{e^{-in\theta}}{2i}\left( e^{2in\theta}-1 \right).$$
Let $\omega =e^{2i\theta}$, then $\omega =e^{\frac{-2ik\pi}{n}},k=0,1,...,n-1$ satisfy the equation $
\omega ^n-1=0$. So we have
$$
e^{2in\theta}-1=\prod_{k=0}^{n-1}{\left( e^{2i\theta}-e^{\frac{-2ik\pi}{n}}\right)}.
$$
Then we obtain that
\begin{align*}
		\sin \left( n\theta \right)
		&=\frac{e^{-in\theta}}{2i}\prod_{k=0}^{n-1}{\left( e^{2i\theta}-e^{\scriptstyle{\frac{-2ik\pi}{n}}} \right)}\\
		&=\frac{e^{-in\theta}}{2i}\prod_{k=0}^{n-1}{\left[ e^{i\theta}e^{\frac{-ik\pi}{n}}\left( e^{i\theta}e^{\frac{ik\pi}{n}}-e^{-i\theta}e^{\frac{-ik\pi}{n}} \right) \right]}\\
		&=\frac{e^{-in\theta}}{2i}\prod_{k=0}^{n-1}{\left[ \left( 2i \right) e^{i\theta}e^{\frac{-ik\pi}{n}}\left( \frac{e^{i\left( \theta +\frac{k\pi}{n} \right)}-e^{-i\left( \scriptstyle{\theta}+\frac{k\pi}{n} \right)}}{2i} \right) \right]}\\
		&=\frac{e^{-in\theta}}{2i}\left( 2i \right) ^ne^{in\theta}e^{\frac{-i\left( n-1 \right) \pi}{2}}\prod_{k=0}^{n-1}{\sin \left( \theta +\frac{k\pi}{n} \right)}\\
		&=2^{n-1}i^{^{n-1}}e^{\frac{-i\left( n-1 \right) \pi}{2}}\prod_{k=0}^{n-1}{\sin \left( \theta +\frac{k\pi}{n} \right)}\\
		&=2^{n-1}e^{\frac{i\left( n-1 \right) \pi}{2}}e^{\frac{-i\left( n-1 \right) \pi}{2}}\prod_{k=0}^{n-1}{\sin \left( \theta +\frac{k\pi}{n} \right)}\\
		&=2^{n-1}\prod_{k=0}^{n-1}{\sin \left( \theta +\frac{k\pi}{n} \right)}.
\end{align*}
That is
$$
\prod_{k=0}^{n-1}{\sin \left( \theta +\frac{k\pi}{n} \right) =\frac{\sin \left( n\theta \right)}{2^{n-1}}}.
$$
By applying the above formula, we have
\begin{align*}
		\prod_{k=1}^r{\sin \alpha _k}
		&=\prod_{k=1}^r{\sin \left( \frac{2k-1}{2r}\pi \right)}\\
		&=\prod_{k=1}^r{\sin \left( \frac{k\pi}{r}-\frac{\pi}{2r} \right)}\\
		&=\left[ \prod_{k=0}^{r-1}{\sin \left( -\frac{\pi}{2r}+\frac{k\pi}{r} \right)} \right] \times \sin \left( \frac{2r-1}{2r}\pi \right) \div \sin \left(-\frac{\pi}{2r} \right) \\
		&=\frac{\sin \left( r\cdot \frac{-\pi}{2r} \right)}{2^{r-1}}\times \sin \left( \pi -\frac{\pi}{2r} \right) \div \sin \left( -\frac{\pi}{2r} \right)\\
		&=\frac{-1}{2^{r-1}}\times \left[ -\sin \left( -\frac{\pi}{2r} \right) \right] \div \sin \left( -\frac{\pi}{2r} \right) \\
		&=\frac{1}{2^{r-1}}.
	\end{align*}
It follows that 
 $$
\prod_{k=1}^r{\sin \alpha _k}=\left( \frac{1}{2} \right) ^{r-1}.
$$
Therefore, we obtain $$
\prod_{k=1}^r{e_{r}^{T}v_k}=\left( -1 \right) ^{\lfloor \frac{r}{2} \rfloor}2^{r-1}.
$$
By Lemma \ref{2.3}, we have rank $W\left( B_1+\overline{D\left( A_n \right) } \right) =r .$
\end{proof}

\begin{Koro}
Suppose that $n$ is even. Then
\begin{center}
$\rank\; W_Q(A_n) =\lceil \frac{n}{2} \rceil .$
\end{center}
\end{Koro}
\begin{proof}
From Lemmas \ref{2.1} and \ref{2.2}, we deduce that rank $W_Q(A_n)$=rank $W( B_1+\overline{D(A_n) })$. Hence rank $W_Q(A_n) =\lceil \frac{n}{2} \rceil $ by Lemma \ref{3.2}. 
\end{proof}

The following lemma is needed for the determinant of  $W( B_1+\overline{D\left( A_n \right) }) $.

 \begin{Lem}\label{3.4}
 \cite{ref14}	It holds that
 $$
 \left| \begin{matrix}
 	1&		1&		\cdots&		1\\
 	2\cos \theta _1&		2\cos \theta _2&		\cdots&		2\cos \theta _q\\
 	2\cos 2\theta _1&       2\cos 2\theta _2&
 	\cdots&     2\cos 2\theta _q\\
 	\vdots&		\vdots&		&		\vdots\\
 	2\cos \left( q-1 \right) \theta _1&		2\cos \left( q-1 \right) \theta _2&		\cdots&		2\cos \left( q-1 \right) \theta _q\\
 \end{matrix} \right|=\prod_{1\leqslant j<i\leqslant q}{\left( 2\cos \theta _i-2\cos \theta _j \right)}.
 $$
 \end{Lem}
 We give a formula for the determinant of $W( B_1+\overline{D(A_n) })$.
 \begin{Prop}
 Suppose that $n$ is even. Then the determinant of $W( B_1+\overline{D(A_n)}) $ is
 $$
 \det  W(B_1+\overline{D(A_n)}) =2^{r-1}.
 $$
\end{Prop}

\begin{proof}

Let $ v_1,...,v_r$ be the eigenvectors of $( B_1+\overline{D(A_n) }) ^T
 $. First, let us calculate the determinant of the matrix
 $
 \left[ \begin{matrix}
 	v_1&		v_2&		\cdots&		v_r\\
 \end{matrix} \right] .
 $
 For each $i=1,...,r-1$, we add the $(i+1)$-th row to the $i$-th row. Then we obtain
 \begin{equation*}
 	\begin{split}
 		&\det \left[ \begin{matrix}
 			v_1&		v_2&		\cdots&		v_r\\
 		\end{matrix} \right]\\
 		&=\left| \begin{matrix}
 			\left( -1 \right) ^{r-1}\left( 1+\sum_{i=1}^{r-1}{2\cos i\alpha _1} \right)&		\left( -1 \right) ^{r-1}\left( 1+\sum_{i=1}^{r-1}{2\cos i\alpha _2} \right)&		\cdots&		\left( -1 \right) ^{r-1}\left( 1+\sum_{i=1}^{r-1}{2\cos i\alpha _r} \right)\\
 			\left( -1 \right) ^{r-2}\left( 1+\sum_{i=1}^{r-2}{2\cos i\alpha _1} \right)&		\left( -1 \right) ^{r-2}\left( 1+\sum_{i=1}^{r-2}{2\cos i\alpha _2} \right)&		\cdots&		\left( -1 \right) ^{r-2}\left( 1+\sum_{i=1}^{r-2}{2\cos i\alpha _r} \right)\\
 			\vdots&		\vdots&		&		\vdots\\
 			-1-2\cos \alpha _1&   -1-2\cos \alpha _2&  \cdots&      -1-2\cos \alpha _r\\
 			1&		1&		\cdots&		1\\
 		\end{matrix} \right|\\
 		&=\left| \begin{matrix}
 			\left( -1 \right) ^{r-1}2\cos \left( r-1 \right) \alpha _1&		\left( -1 \right) ^{r-1}2\cos \left( r-1 \right) \alpha _2&		\cdots&		\left( -1 \right) ^{r-1}2\cos \left( r-1 \right) \alpha _r\\
 			\left( -1 \right) ^{r-2}2\cos \left( r-2 \right) \alpha _1&		\left( -1 \right) ^{r-2}2\cos \left( r-2 \right) \alpha _2&		\cdots&		\left( -1 \right) ^{r-2}2\cos \left( r-2 \right) \alpha _r\\
 			\vdots&		\vdots&		&		\vdots\\
 			-2\cos \alpha _1&  -2\cos \alpha _2&
 			\cdots&      -2\cos \alpha _r\\
 			1&		1&		\cdots&		1\\
 		\end{matrix} \right|\\
 		&=\left( -1 \right) ^{\frac{\left( r-1 \right) r}{2}}\left| \begin{matrix}
 			2\cos \left( r-1 \right) \alpha _1&		2\cos \left( r-1 \right) \alpha _2&		\cdots&		2\cos \left( r-1 \right) \alpha _r\\
 			2\cos \left( r-2 \right) \alpha _1&		2\cos \left( r-2 \right) \alpha _2&		\cdots&		2\cos \left( r-2 \right) \alpha _r\\
 			\vdots&		\vdots&		&		\vdots\\
 			2\cos \alpha _1&  2\cos \alpha _2&
 			\cdots&      2\cos \alpha _r\\
 			1&		1&		\cdots&		1\\
 		\end{matrix} \right|\\
 		&=\left( -1 \right) ^{\frac{\left( r-1 \right) r}{2}+\lfloor \frac{r}{2} \rfloor}\left| \begin{matrix}
 				1&		1&		\cdots&		1\\
 				2\cos \alpha _1&  2\cos \alpha _2&
 				\cdots&      2\cos \alpha _r\\
 				\vdots&		\vdots&		&		\vdots\\
 			2\cos \left( r-2 \right) \alpha _1&		2\cos \left( r-2 \right) \alpha _2&		\cdots&		2\cos \left( r-2 \right) \alpha _r\\
 			2\cos \left( r-1 \right) \alpha _1&		2\cos \left( r-1 \right) \alpha _2&		\cdots&		2\cos \left( r-1 \right) \alpha _r\\
 		\end{matrix} \right|.
 	\end{split}
 \end{equation*}
By Lemma \ref{3.4}, we have
\begin{equation*}
\begin{split}
	&\det  \left[ \begin{matrix}
		v_1&		v_2&		\cdots&		v_r\\
	\end{matrix} \right] \\
	&=\left( -1 \right) ^{\frac{\left( r-1 \right) r}{2}+\lfloor \frac{r}{2} \rfloor}\prod_{1\leqslant j<i\leqslant r}{\left( 2\cos \alpha _i-2\cos \alpha _j \right)}\\
	&=\left( -1 \right) ^{\lfloor \frac{r}{2} \rfloor}\prod_{1\leqslant j<i\leqslant r}{\left( -2\cos \alpha _i+2\cos \alpha _j \right)}.
\end{split}
\end{equation*}
Note that $$
-2\cos \alpha _i+2\cos \alpha _j=\left( 2-2\cos \alpha _i \right) -\left( 2-2\cos \alpha _j \right) =\lambda _i-\lambda _j,
$$
where we recall that $\lambda _k=2-2\cos \alpha _k$.
So we have
$$\det  \left[ \begin{matrix}
	v_1&		v_2&		\cdots&		v_r\\
\end{matrix} \right] =\left( -1 \right) ^{\lfloor \frac{r}{2} \rfloor}\prod_{1\leqslant j<i\leqslant r}{\left(\lambda _i-\lambda _j \right)}.$$
By Lemma \ref{2.3} and Lemma \ref{3.2}, the determinant of $W\left( B_1+\overline{D\left( A_n \right) } \right) $ is $2^{r-1}$.
\end{proof}

Now, we can formulate the Smith normal form of $
W_Q\left( A_n \right) $ when $n$ is even.
\begin{Koro}\label{3.6}
Suppose that $n$ is even.Then the Smith normal form of $W_Q\left( A_n \right) $ is
$$
\mathrm{diag}\left( \underset{r=\lceil \frac{n}{2} \rceil}{\underbrace{1,2,2,...,2}},0,...,0 \right),
$$
where $r$ is the rank of $W_Q\left( A_n \right) $.
 \end{Koro}
 \begin{proof}
When $n$ is even, by Lemma \ref{2.2}, we have
$$
\det \overline{W_Q\left( A_n \right) }=\det W\left( B_1+\overline{D\left( A_n \right) } \right)= 2^{r-1}.
$$
By the definition of invariant factors, the Smith normal form of $\overline{W_Q\left( A_n \right) }$is
$$
\mathrm{diag}\left( 1,2,2,...,2 \right).
$$
Thus it follows from Lemma \ref{2.1} that $W_Q\left( A_n \right)$ has the Smith normal form
	$$
	\mathrm{diag}\left( \underset{r=\lceil \frac{n}{2} \rceil}{\underbrace{1,2,2,...,2}},0,...,0 \right).
	$$
 \end{proof}

\subsection{Smith normal form of $W_Q\left( A_n \right) $ for $n$ is odd
}
\begin{Prop}\label{3.7}
	Let $\mu _k=2-2\cos \beta _k$ and let
	$$
	w_k=\left[ \begin{array}{c}
		\left( -1 \right) ^{r-1}2\cos \left( r-1 \right) \beta _k\\
		\left( -1 \right) ^{r-2}2\cos \left( r-2 \right) \beta _k\\
		\vdots\\
		-2\cos \beta _k\\
		1\\
	\end{array} \right] ,
	$$
	where $\beta _k=\frac{2k-2}{2r-1}\pi $  for $k=1,...,r$. Then $w_k$ is an eigenvector of $
	\left( B_2+\overline{D\left( A_n \right) } \right) ^T
	$ corresponding to the eigenvalue $\mu _k$ for $k=1,...,r$.
\end{Prop}
 \begin{proof}
 It is enough to show that $(B_2+\overline{D(A_n) }) ^Tw_k=\mu _kw_k
$ for $k=1,...,r$. First, we show that the first entries of $( B_2+\overline{D(A_n) }) ^Tw_k$ and $\mu _kw_k$ are the same, that is,
$$
(-1) ^{r-1}2\cos(r-1) \beta _k+(-1) ^{r-2}2\cos(r-2) \beta _k=(-1) ^{r-1}(2-2\cos\beta _k) 2\cos(r-1) \beta _k.
$$
The first entry of $(B_2+\overline{D(A_n) }) ^Tw_k$ is
\begin{equation*}
	\begin{split}
		&(-1) ^{r-1}2\cos (r-1) \beta _k+(-1) ^{r-2}2\cos(r-2) \beta _k\\
		&=(-1) ^{r-1}2\cos(r-1) \beta _k-(-1) ^{r-1}2\cos(r-2) \beta _k\\
		&=2(-1) ^{r-1}[ \cos (r-1) \beta _k-\cos(r-2) \beta _k ].
	\end{split}
\end{equation*}
The first entry of $\mu _kw_k$ is
\begin{equation*}
	\begin{split}
	&(-1) ^{r-1}(2-2\cos \beta _k) 2\cos(r-1) \beta _k\\
	&=2(-1) ^{r-1}[ 2\cos(r-1) \beta _k-2\cos(r-1) \beta _k\cos \beta _k ].
	\end{split}
\end{equation*}
\\By applying the formula $2\cos \alpha \cos \beta =\cos ( \alpha +\beta) +\cos(\alpha -\beta)$ , we obtain
\begin{equation*}
	\begin{split}
	&(-1) ^{r-1}( 2-2\cos \beta _k) 2\cos(r-1) \beta _k\\
	&=2(-1) ^{r-1}[ 2\cos( r-1) \beta _k-(\cos r\beta _k+\cos(r-2) \beta _k) ].	
	\end{split}
\end{equation*}
Since $\cos(2r-1) \beta _k=\cos 2(k-1) \pi =1$ and $\sin(2r-1) \beta _k=\sin 2(k-1) \pi =0$, we have
 \begin{equation*}
 	\begin{split}
 		&2\cos(r-1) \beta _k\\	
 		&=2\cos (2r-1) \beta _k\cos( r-1) \beta _k\\
 		&=\cos[(2r-1) +( r-1)] \beta _k+\cos[(2r-1) -(r-1)] \beta _k\\
 		&=[\cos(r-1) \beta _k\cos(2r-1) \beta _k-\sin(r-1) \beta _k\sin(2r-1) \beta _k]+\cos r\beta _k\\
 		&=\cos(r-1) \beta _k+\cos r\beta _k.
 	\end{split}
 \end{equation*}
Therefore, we obtain that the first entry of $\mu _kw_k$ is
\begin{equation*}
	\begin{split}
		&(-1) ^{r-1}(2-2\cos \beta _k) 2\cos(r-1) \beta _k\\
	&=2(-1) ^{r-1}[\cos(r-1) \beta _k+\cos r\beta _k-( \cos r\beta _k+\cos(r-2) \beta _k)] \\
	&=2(-1) ^{r-1}[\cos(r-1) \beta _k-\cos(r-2) \beta _k] .
	\end{split}
\end{equation*}
Hence the first entries of $( B_2+\overline{D(A_n) }) ^Tw_k$ and $\mu _kw_k$ are the same.

Now, we prove that the $j$-th entries of $(B_2+\overline{D(A_n) }) ^Tw_k$ and $\mu _kw_k$ are equal for $j=2,...,r-1$. The $j$-th entry of $\mu _kw_k$ is
$$
(2-2\cos \beta _k) \cdot (-1) ^{r-j}2\cos( r-j) \beta _k.
$$
The $j$-th entry of $( B_2+\overline{D(A_n) }) ^Tw_k$  is
\begin{equation*}
	\begin{split}
		&(-1) ^{r+1-j}2\cos( r+1-j) \beta _k+2(-1) ^{r-j}\cdot 2\cos(r-j) \beta _k+(-1) ^{r-1-j}2\cos(r-1-j) \beta _k\\
		&=(-1) ^{r-j}[ -2\cos( r+1-j) \beta _k+4\cos(r-j) \beta _k-2\cos(r-1-j) \beta _k] \\
		&=2(-1) ^{r-j}[ 2\cos(r-j) \beta _k-( \cos( r+1-j) \beta _k+\cos(r-1-j) \beta _k)] \\
		&=2(-1) ^{r-j}[ 2\cos( r-j) \beta _k-2\cos(r-j) \beta _k\cos \beta _k] \\
		&=2(-1) ^{r-j}(1-\cos \beta _k) \cdot 2\cos( r-j) \beta _k.
	\end{split}
\end{equation*}
Thus the $j$-th entries of $\left( B_2+\overline{D\left( A_n \right) } \right) ^Tw_k$ and $\mu _kw_k$ are the same for $j=2,...,r-1$.
Since
$$
1\times \left( -2\cos \beta _k \right) +2\times 1=2-2\cos \beta _k=\left( 2-2\cos \beta _k \right) \times 1,
$$
the last entries of $\left( B_2+\overline{D\left( A_n \right) } \right) ^Tw_k$ and $\mu _kw_k$ are also equal. Therefore, $( B_2+\overline{D(A_n) }) ^Tw_k=\mu _kw_k$ for $k=1,...,r$.
\end{proof}

In order to find the determinant of $W( B_2+\overline{D(A_n) }) $, we first compute the product $\prod_{k=1}^r{e_{r}^{T}w_k}$.

\begin{Lem}\label{3.8}
Let $e_r$ be the all-ones vector of dimension $r$, and let $w_k$ be as above. Then
	$$
	\prod_{k=1}^r{e_{r}^{T}w_k}=\left( -1 \right) ^{\lfloor \frac{r}{2} \rfloor}2^{r-1}.
	$$
Furthermore, $\rank\;W( B_2+\overline{D(A_n) }) =r$.
\end{Lem}

\begin{proof} We have the following 
$$
e_{r}^{T}w_k=\left( -1 \right) ^{r-1}2\cos \left( r-1 \right) \beta _k+\left( -1 \right) ^{r-2}2\cos \left( r-2 \right) \beta _k+\cdots +2\cos 2\beta _k-2\cos \beta _k+1.
$$
If $r$ is odd, by $
2\cos \alpha \sin \beta =\sin( \alpha +\beta) -\sin(\alpha -\beta)
$, it follows that
	\begin{equation*}
		\begin{split}
			&\sin \beta _ke_{r}^{T}w_k\\
			&=\sin \beta _k\left[ 2\cos \left( r-1 \right) \beta _k-2\cos \left( r-2 \right) \beta _k+\cdots +2\cos 2\beta _k-2\cos \beta _k+1 \right] \\
			&=\left( \sin r\beta _k-\sin \left( r-2 \right) \beta _k \right) -\left( \sin \left( r-1 \right) \beta _k-\sin \left( r-3 \right) \beta _k \right) +\cdots +\left( \sin 3\beta _k-\sin \beta _k \right) -\sin 2\beta _k+\sin \beta _k\\
		&=\sin r\beta _k-\sin \left( r-1 \right) \beta _k.
\end{split}
\end{equation*}
Note that $\beta _k=\frac{2k-2}{2r-1}\pi \in \left( 0,\pi \right) $ for $k=1,...,r$, so $\sin \beta _k\ne 0$. Hence
$$
\prod_{k=1}^r{e_{r}^{T}w_k=\prod_{k=1}^r{\frac{\sin r\beta _k-\sin \left( r-1 \right) \beta _k}{\sin \beta _k}}}.
$$
Similarly, if $r$ is even, then
\begin{equation*}
\begin{split}
&\sin \beta _ke_{r}^{T}w_k\\
&=\sin \beta _k\left[ -2\cos \left( r-1 \right) \beta _k+2\cos \left( r-2 \right) \beta _k-\cdots+2\cos 2\beta _k-2\cos \beta _k+1 \right] \\
&=-\left( \sin r\beta _k-\sin \left( r-2 \right) \beta _k \right) +\left( \sin \left( r-1 \right) \beta _k-\sin \left( r-3 \right) \beta _k \right) -\cdots +\left( \sin 3\beta _k-\sin \beta _k \right) -\sin 2\beta _k+\sin \beta _k\\
&=-\sin r\beta _k+\sin \left( r-1 \right) \beta _k.
\end{split}
\end{equation*}
In this case, we have
$$
\prod_{k=1}^r{e_{r}^{T}w_k=\prod_{k=1}^r{\frac{-\sin r\beta _k+\sin \left( r-1 \right) \beta _k}{\sin \beta _k}}}=\left( -1 \right) ^r\prod_{k=1}^r{\frac{\sin r\beta _k-\sin \left( r-1 \right) \beta _k}{\sin \beta _k}=\prod_{k=1}^r{\frac{\sin r\beta _k-\sin \left( r-1 \right) \beta _k}{\sin \beta _k}}}.
$$
Thus the the following equation holds $$\prod_{k=1}^r{e_{r}^{T}w_k=\prod_{k=1}^r{\frac{\sin r\beta _k-\sin \left( r-1 \right) \beta _k}{\sin \beta _k}}}.$$
By using $\sin \alpha -\sin \beta =2\cos \left( \frac{\alpha +\beta}{2} \right) \sin \left( \frac{\alpha -\beta}{2} \right)$ and $
\sin \alpha =2\sin \frac{\alpha}{2}\cos \frac{\alpha}{2}
$, we have
$$
\prod_{k=1}^r{e_{r}^{T}w_k=\prod_{k=1}^r{\frac{2\cos \left( \frac{2r-1}{2} \right) \beta _k\sin \frac{1}{2}\beta _k}{2\sin \frac{1}{2}\beta _k\cos \frac{1}{2}\beta _k}=\prod_{k=1}^r{\textstyle{\frac{\cos \left( \frac{2r-1}{2} \right) \beta _k}{\cos \frac{1}{2}\beta _k}}}}}.
$$
Since $\cos \left( \frac{2r-1}{2} \right) \beta _k=\cos \left( k-1 \right) \pi $, we obtain
$$
\prod_{k=1}^r{e_{r}^{T}w_k=\textstyle{\frac{\prod_{k=1}^r{\cos \left( k-1 \right) \pi}}{\prod_{k=1}^r{\cos \frac{1}{2}\beta _k}}}}.
$$
Consequently, we see that
\begin{center}
$
\prod_{k=1}^r{\cos \left( k-1 \right) \pi}=\left( -1 \right) ^{\lfloor \frac{r}{2} \rfloor}
$ and $\prod_{k=1}^r{\cos \frac{1}{2}\beta _k}=\left( \frac{1}{2} \right) ^{r-1}.
$
\end{center}
The specific calculations are given below. We first calculate $\prod_{k=1}^r{\cos \left( k-1 \right) \pi}$. Since 
$$
\cos k\pi=\begin{cases}
1,&\text{if $k$ is even}\\
-1,&\text{if $k$ is odd}\\
\end{cases},
$$
we have
$$\prod_{k=1}^r{\cos \left( k-1 \right) \pi}=\prod_{k=0}^{r-1}{\cos k\pi}=\left( -1 \right) ^{\lfloor \frac{r}{2} \rfloor}.
$$
Then we calculate $\prod_{k=1}^r{\cos \frac{1}{2}\beta _k}.$ First, we show that
$$\prod_{k=1}^m{\cos \left( \frac{k\pi}{2m+1} \right)}=\frac{1}{2^m}.$$
Note that $$
x^{2m+1}-1=\left( x-1 \right) \left( x-\xi \right) \left( x-\xi ^2 \right) \cdots \left( x-\xi ^{2m} \right) ,\xi =e^{i\frac{2\pi}{2m+1}}.
$$
Considering that the root $\xi ^k$ and its conjugate root $\xi ^{2m+1-k}$ ($k=1,..,m$) satisfy
$$
\left( x-\xi ^k \right) \left( x-\xi ^{2m+1-k} \right) =x^2-2x\cos \frac{2k\pi}{2m+1}+1,
$$
we have
\begin{align*}
	x^{2m+1}-1
	&=\left( x-1 \right) \left[ \left( x-\xi \right) \left( x-\xi ^{2m} \right) \right] \cdots \left[ \left( x-\xi ^m \right) \left( x-\xi ^{m+1} \right) \right] \\
	&=\left( x-1 \right) \left( x^2-2x\cos \frac{2\pi}{2m+1}+1 \right) \cdots \left( x^2-2x\cos \frac{2m\pi}{2m+1}+1 \right) \\
	&=\left( x-1 \right) \prod_{k=1}^m{\left( x^2-2x\cos \frac{2k\pi}{2m+1}+1 \right)}.
\end{align*}
Let $x=-1$. Then we have
$$
-2=-2\prod_{k=1}^m{\left( 2+2\cos \frac{2k\pi}{2m+1} \right)}=-2\cdot 2^m\prod_{k=1}^m{\left( 1+\cos \frac{2k\pi}{2m+1} \right)}.
$$
By applying the formula $\cos 2\alpha =2\cos \alpha ^2-1$,
we have
$$
\frac{1}{2^m}=\prod_{k=1}^m{2\cos ^2\frac{k\pi}{2m+1}}=2^m\prod_{k=1}^m{\cos ^2\frac{k\pi}{2m+1}}.
$$
That is
$$
\frac{1}{2^{2m}}=\prod_{k=1}^m{\cos ^2\frac{k\pi}{2m+1}}.
$$
So we have
$$\prod_{k=1}^m{\cos \left( \frac{k\pi}{2m+1} \right)}=\frac{1}{2^m}.$$
By applying the above formula, we obtain
$$
\prod_{k=1}^r{\cos \frac{1}{2}\beta _k}
=\prod_{k=1}^r{\cos \left( \frac{1}{2}\cdot\frac{2k-2}{2r-1}\pi \right)}
=\prod_{k=1}^r{\cos \left( \frac{k-1}{2r-1}\pi \right)}
=\prod_{k=0}^{r-1}{\cos \frac{k\pi}{2r-1}}
=\prod_{k=1}^{r-1}{\cos \frac{k\pi}{2\left( r-1 \right) +1}}
=\frac{1}{2^{r-1}}.
$$
Hence we have
$$
\prod_{k=1}^r{\cos \frac{1}{2}\beta _k}=\left( \frac{1}{2} \right) ^{r-1}.
$$
Therefore, we obtain 
$$
\prod_{k=1}^r{e_{r}^{T}w_k}=\left( -1 \right) ^{\lfloor \frac{r}{2} \rfloor}2^{r-1}.
$$
By Lemma \ref{2.3}, we have $\rank\; W\left( B_2+\overline{D\left( A_n \right) } \right) =r$.
\end{proof}

\begin{Koro}
Suppose that $n$ is odd. Then
\begin{center}
$\rank\; W_Q\left( A_n \right) =\lceil \frac{n}{2} \rceil.$
\end{center}
\end{Koro}

\begin{proof}
From Lemmas \ref{2.1} and \ref{2.2}, we deduce that rank $W_Q( A_n)=\rank W( B_2+\overline{D\left( A_n \right) })$. Hence
	$\rank W_Q(A_n) =\lceil \frac{n}{2} \rceil$ by Lemma \ref{3.8}.   
\end{proof}

We give a formula for the determinant of $W( B_2+\overline{D( A_n)})$.

\begin{Prop}
Suppose that $n$ is odd. Then the determinant of $W( B_2+\overline{D( A_n) }) $ is
	$$
	\det  W( B_2+\overline{D( A_n)}) =2^{r-1}.
	$$
 \end{Prop}
 
 \begin{proof}
 Let $ w_1,...,w_r$ be the eigenvectors of $
	\left( B_2+\overline{D\left( A_n \right) } \right) ^T
	$. First, let us calculate the determinant of the matrix
	$\left[ \begin{matrix}
		w_1&		w_2&		\cdots&		w_r\\
        \end{matrix} \right] .$
	\begin{equation*}
		\begin{split}
			&\det \left[ \begin{matrix}
				w_1&		w_2&		\cdots&		w_r\\
			\end{matrix} \right]\\
			&=\left| \begin{matrix}
				\left( -1 \right) ^{r-1}2\cos \left( r-1 \right) \beta _1&		\left( -1 \right) ^{r-1}2\cos \left( r-1 \right) \beta _2&		\cdots&		\left( -1 \right) ^{r-1}2\cos \left( r-1 \right) \beta _r\\
				\left( -1 \right) ^{r-2}2\cos \left( r-2 \right) \beta _1&		\left( -1 \right) ^{r-2}2\cos \left( r-2 \right) \beta _2&		\cdots&		\left( -1 \right) ^{r-2}2\cos \left( r-2 \right) \beta _r\\
				\vdots&		\vdots&		&		\vdots\\
				-2\cos \beta _1&  -2\cos \beta _2&
				\cdots&      -2\cos \beta _r\\
				1&		1&		\cdots&		1\\
			\end{matrix} \right|\\
			&=\left( -1 \right) ^{\frac{\left( r-1 \right) r}{2}}\left| \begin{matrix}
				2\cos \left( r-1 \right) \beta _1&		2\cos \left( r-1 \right) \beta _2&		\cdots&		2\cos \left( r-1 \right) \beta _r\\
				2\cos \left( r-2 \right) \beta _1&		2\cos \left( r-2 \right) \beta _2&		\cdots&		2\cos \left( r-2 \right) \beta _r\\
				\vdots&		\vdots&		&		\vdots\\
				2\cos \beta _1&  2\cos \beta _2&
				\cdots&      2\cos \beta _r\\
				1&		1&		\cdots&		1\\
			\end{matrix} \right|\\
			&=\left( -1 \right) ^{\frac{\left( r-1 \right) r}{2}+\lfloor \frac{r}{2} \rfloor}\left| \begin{matrix}
				1&		1&		\cdots&		1\\
				2\cos \beta _1&  2\cos \beta _2&
				\cdots&      2\cos \beta _r\\
				\vdots&		\vdots&		&		\vdots\\
				2\cos \left( r-2 \right) \beta _1&		2\cos \left( r-2 \right) \beta _2&		\cdots&		2\cos \left( r-2 \right) \beta _r\\
				2\cos \left( r-1 \right) \beta _1&		2\cos \left( r-1 \right) \beta _2&		\cdots&		2\cos \left( r-1 \right) \beta _r\\
			\end{matrix} \right|.
		\end{split}
	\end{equation*}
By Lemma \ref{3.4}, we have
\begin{equation*}
\begin{split}
&\det  \left[ \begin{matrix}
w_1&		w_2&		\cdots&		w_r\\
\end{matrix} \right] \\
&=\left( -1 \right) ^{\frac{\left( r-1 \right) r}{2}+\lfloor \frac{r}{2} \rfloor}\prod_{1\leqslant j<i\leqslant r}{\left( 2\cos \beta _i-2\cos \beta _j \right)}\\
&=\left( -1 \right) ^{\lfloor \frac{r}{2} \rfloor}\prod_{1\leqslant j<i\leqslant r}{\left( -2\cos \beta _i+2\cos \beta _j \right)}.
\end{split}
\end{equation*}
Note that 
$$
-2\cos \beta _i+2\cos \beta _j=\left( 2-2\cos \beta _i \right) -\left( 2-2\cos \beta _j \right) =\mu _i-\mu _j,
$$
where we recall that $\mu _k=2-2\cos \beta _k$.
So we have
$$
\det  \left[ \begin{matrix}
		w_1&		w_2&		\cdots&		w_r\\
	\end{matrix} \right] =\left( -1 \right) ^{\lfloor \frac{r}{2} \rfloor}\prod_{1\leqslant j<i\leqslant r}{\left(\mu _i-\mu _j \right)}.
$$
By Lemmas \ref{2.3} and \ref{3.8}, the determinant of $W\left( B_2+\overline{D\left( A_n \right) } \right) $ is $2^{r-1}$.
 \end{proof}

Now, we find the Smith normal form of $W_Q\left( A_n \right) $ when $n$ is odd.
\begin{Koro}
Suppose that $n$ is odd. Then the Smith normal form of $W_Q\left( A_n \right) $ is
$$
\mathrm{diag}\left( \underset{r=\lceil \frac{n}{2} \rceil}{\underbrace{1,2,2,...,2}},0,...,0 \right),
$$
where $r$ is the rank of $W_Q\left( A_n \right) $.
 \end{Koro}	
 
\begin{proof}
The proof is similar to that of Corollary \ref{3.6}. When $n$ is odd, by Lemma \ref{2.2}, we have
$$
\det \overline{W_Q\left( A_n \right) }=\det W\left( B_2+\overline{D\left( A_n \right) } \right)= 2^{r-1}.$$
Thus, by the definition of invariant factors and Lemma \ref{2.1}, the Smith normal form of $W_Q\left( A_n \right)$ is
$$
\mathrm{diag}\left( \underset{r=\lceil \frac{n}{2} \rceil}{\underbrace{1,2,2,...,2}},0,...,0 \right).
$$
\end{proof}

\section{Concluding remarks\label{con}}

For the Dynkin graph $A_n$ with $n$ vertices, the ranks of the $Q$-walk matrix $W_Q\left( A_n \right)$ are the same for both even and odd values of $n$. Thus, we may write the rank of the $Q$-walk matrix $W_Q\left( A_n \right)$ as follows:
\begin{center}
rank $W_Q\left( A_n \right) =\lceil \frac{n}{2} \rceil $.
\end{center}
The Smith normal form of $W_Q\left( A_n \right)$ are the same for both even and odd values of $n$. Thus, we may write the Smith normal form of $W_Q\left( A_n \right)$ as follows:
$$
\mathrm{diag}\left( \underset{\lceil \frac{n}{2} \rceil}{\underbrace{1,2,2,...,2}},0,...,0 \right).
$$

\bigskip
{\bf Acknowledgements.} The authors are grateful to the referees for their careful reading, valuable comments and suggestions. Shengyong Pan is supported by Beijing Natural Science Foundation (1262017,1252011) and the Fundamental Research Funds for the Central Universities of Beijing Jiaotong University (2024JBMC001)

{\bf Conflicts of interest}\\
The authors declared that they have no conflicts of interest to this work.


\begin{thebibliography}{99}



\bibitem{ref7} I. Assem, D. Simson, A. Skowro\'{n}ski, Elements of Representation Theory of Associative Algebras I: Techniques of Representation Theory, 2006.

\bibitem{ref10} D. Cvetković, P. Rowlinson, S. Simić, An Introduction to the Theory of Graph Spectra, Cambridge University Press, Cambridge, 2010.

\bibitem{ref8} M. A. Dokuchaev, N. M. Gubareni, V. M. Futorny, M. A. Khibina, V. V. Kirichenko, Dynkin diagrams and spectra of graphs, S\~{a}o Paulo J. Math. Sci. 7 (2013), 83-104.

\bibitem{ref4} C. D. Fontanals, Notes on G(2): The Lie algebra and the Lie group, Differ. Geom. Appl. 57 (2018), 23-74.

\bibitem{ref11} E. M. Hagos, Some results on graph spectra, Linear Algebra Appl. 356 (2002), 103-111.

\bibitem{ref5} A. Hasi\'{c}, Introduction to Lie Algebras and Their Representations, Advances in Linear Algebra and Matrix Theory, 2021.

\bibitem{ref2} L. Huang, Y. Xu, H. Zhang, The smith normal form of the walk matrix of the Dynkin graph $A_n$, Linear Algebra Appl. 698 (2024), 26-39.

\bibitem{ref6} L. A. Nazarova, Representations of quivers of infinite type, Mathematics of the USSR-Izvestiya, 7(4)(1973), 752-791.

\bibitem{ref9} V. L. Ostrovs'kyi, Y. S. Samoilenko, On spectral theorems for families of linearly connected self-adjoint operators with given spectra associated with extended Dynkin graphs, Ukrainian Mathematical Journal 58 (2006), 1768-1785.

\bibitem{QYJ} L. Qiu, Z. Yuan, Y. Ji, On the Smith normal form of $Q$-walk matrix, Linear Algebra Appl. 729 (2026), 187-202.

\bibitem{ref12} P. Rowlinson, The main eigenvalues of a graph: a survey, Applicable Analysis and Discrete Mathematics 1(2)(2007), 455-471.

\bibitem{ref3} S. Moon, S. Park, The Smith normal form of the walk matrix of the extended Dynkin graph $\tilde{D}_n$, Linear Algebra Appl. 678 (2023), 169-190.

\bibitem{ref1} Z. Yan, L. Mao, W. Wang, On the determinant of the Q-walk matrix of rooted product with a path, Bulletin of the Malaysian Mathematical Sciences Society 47(2024), 174-174.

\bibitem{ref13} W. Wang, C. Wang, S. Guo, On the walk matrix of the Dynkin graph $D_n$, Linear Algebra Appl. 653 (2022), 193-206.

\bibitem{ref14} W. Wang, The Smith normal form of the walk matrix of the Dynkin graph $D_n$ for $n\equiv 0\pmod{4}$, Linear Algebra Appl. 671(2023), 121-134.

\bibitem{ref15} W. Wang, Z. Yan, L. Mao, Proof of a conjecture on the determinant of the walk matrix of rooted product with a path, Linear Multilinear Algebra 72(2024), No. 5, 828-840.









\end{thebibliography}
\end{document}